\documentclass[a4paper,12pt]{amsart}

\usepackage[english]{babel}
\usepackage{color, mathtools}
\usepackage{graphicx}

\usepackage[alphabetic]{amsrefs}

\usepackage{etoolbox}
\apptocmd{\sloppy}{\hbadness 10000\relax}{}{}

\newtheorem{theorem}{Theorem}[section]
\newtheorem{corollary}[theorem]{Corollary}
\newtheorem{question}[theorem]{Question}
\newtheorem{lemma}[theorem]{Lemma}
\newtheorem{proposition}[theorem]{Proposition}
\newtheorem{example}[theorem]{Example}
\theoremstyle{definition}

\theoremstyle{remark}

\newcommand{\ZZ}{\mathbb{Z}}
\newcommand{\RR}{\mathbb{R}}
\newcommand{\CC}{\mathbb{C}}

\newcommand{\KK}{\mathbb{K}}

\newcommand{\Ve}{\operatorname{Vec}}
\newcommand{\Ch}{\operatorname{char}}

\title[bracket width for Lie algebras of vector fields is finite]{The bracket width for Lie algebras of vector fields is finite}

\author{Rafael B. Andrist}
\address{Faculty of Mathematics and Physics \\
University of Ljubljana \\
Ljubljana, Slovenia}
\email{rafael-benedikt.andrist@fmf.uni-lj.si}
\author{Andriy Regeta}
\address{Dipartimento di Matematica ``Tullio Levi-Civita'',
Universit\'a di Padova, Via Trieste 63, 35121 Padova \\ 
 Italy}
\email{andriy.regeta@math.unipd.it}
\begin{document}

\begin{abstract}
We prove that the bracket width of the Lie algebra of vector fields on any smooth affine algebraic variety of dimension $n$ is at most $(n+1)^2$. We give improved bounds for some families of $\mathbb{C}^*$-varieties, in particular for $\mathrm{SL}_n(\mathbb{C})$ and for the Koras--Russell cubic threefold. 
\end{abstract}

\maketitle

\section{Introduction}

The notion of the bracket width of a Lie algebra $L$ over a field $\mathbb{K}$ was defined in  \cite{MR4353334}. We define it as supremum of the lengths $\ell(z)$, where $z$ is running over the derived algebra $[L,L]$ and $\ell(z)$
is defined as the smallest number $m$ of Lie brackets $[x_i,y_i]$ needed to represent $z$ in the form
\[
z =\sum_{i=1}^{m}
[x_i, y_i].
\]
The bracket width is used in studying different aspects of Lie algebras,  see for example \cite{MR3601334} and \cite{LT13}.

The bracket width of a finite-dimensional complex simple Lie algebra is known to be one, see \cite{Br63}. In general, the bracket width of a finite-dimensional simple Lie algebra over an arbitrary field is at most two (see \cite{BN11}), but there is no known example of a finite-dimensional simple Lie algebra of bracket width exactly two.
The bracket width of current Lie algebras $\mathfrak g\otimes_{\mathbb{K}} A$ was studied in  \cite{KMR2025}, where $\mathfrak g$ is a simple finite-dimensional Lie algebra and $A$ is a commutative associative $\mathbb{K}$-algebra with the identity. It is shown there that the bracket width of such Lie algebras is at most two -- and both numbers, one and two -- appear as a bracket width of such algebras.

The first example of a simple Lie algebra with the bracket width strictly bigger than one was found in \cite{MR4353334} among Lie algebras of vector fields $\Ve(X)$ on smooth affine curves $X$ with 
trivial tangent bundle. Moreover, in \cite{MR4636914} it is proved that the bracket width of the latter Lie algebra is at most three. 
In the current paper we generalize this result and provide a bound for the bracket width of the Lie algebra $\Ve(X)$ for an arbitrary affine smooth variety $X$. We restrict ourselves to smooth affine varieties since the Lie algebra of vector fields on an affine variety $X$ is simple if and only if $X$ is smooth, see \cite{Sie96}*{Proposition 1}.

\begin{theorem}\label{mainthm}
 Let $X$ be a smooth affine $n$-dimensional variety over an algebraically closed field $\KK$ with $\mathrm{char}(\KK) \neq 2$. Then the bracket width of $\Ve(X)$ is at most $(n+1)^2$.
\end{theorem}

It is worth noting that simple Lie algebras with infinite bracket width exist. This was recently shown
in \cite{902976466}*{Theorem 4.12} (based on \cite{Robert}) where the authors constructed such a Lie algebra which is coming from $C^*$-algebras. Hence, Theorem \ref{mainthm} shows that simple Lie algebras of vector fields on affine varieties are principally different from those Lie algebras which may arise from $C^*$-algebras.

In Proposition \ref{boundforC*var}
 we give improved bounds (compared to the ones given in Theorem \ref{mainthm}) for some families of $\mathbb{K}^*$-varieties, where $\mathbb{K}$ is an algebraically closed field of characteristic zero. In particular, we prove (see Example \ref{bracketwidthSLn}) that the bracket width of $\Ve(\mathrm{SL}_n(\mathbb{K}))$ is at most
\[
n^2 + n \sim O(\dim \mathrm{SL}_n(\mathbb{K})).
\]
Furthermore,  we show that the bracket width of $\Ve(\mathrm{SL}_2(\mathbb{K}))$ is at most $3 < 2^2+2$ (see Proposition \ref{SL2generators})
as $\Ve(\mathrm{SL}_2(\mathbb{K}))$ 
is generated by three elements (see Corollary \ref{bracketboundSL2}) and the number of generators of a perfect Lie algebra is an upper bound for the bracket width by Lemma \ref{lem-generatorbound}.
In view of the latter lemma and of Theorem \ref{mainthm} it is of interest to know if $\Ve(X)$ is finitely generated as a Lie algebra for a (smooth) affine variety. Note that in case $X$ is a curve, this is known due to a recent result \cite{Mat22}*{Theorem 15}.

In all examples of simple Lie algebras of finite bracket width where we know the bracket width and the number of generators precisely, the bracket width is strictly smaller than the number of generators. Hence, we have the following question.

\begin{question}
Let $X$ be a smooth affine algebraic variety of positive dimension. Is the bracket width of $\Ve(X)$ always strictly smaller than the number of generators of the Lie algebra $\Ve(X)$?
\end{question}

\textbf{Acknowledgement:} 
The authors would like to thank Boris Kunyavskii for pointing out the reference \cite{Mat22} and a mistake in Example \ref{korasrussel} in an earlier draft.

\section{Smooth affine varieties}

A simple, yet extremely useful ingredient in our proof is the following formula due to Kaliman and Kutzschebauch. Its main virtue is that the right-hand side does not involve a commutator and is a multiple of the functions $f$ and $g$. Moreover, the expression is linear in $h$. Note that it vanishes identically in characteristic $2$ if $\delta = \eta$.
\begin{lemma}[Proof of \cite{MR2385667}*{Theorem 2}]
Let $\delta, \eta$ be vector fields on an affine variety $X$. Let $f,g,h \in \mathbb{K}[X]$. Then
\begin{equation}
\label{KK-formula} 
\tag{KK}
    [f \delta, h g \eta] - [h f \delta,  g \eta] = f g \cdot (\delta(h)\eta + \eta(h) \delta)
\end{equation}
\end{lemma}
\begin{proof}
    The proof is a straight-forward computation.
\end{proof}

We need the following, well-known facts about coherent sheaves on affine varieties which can be found e.g.\ in the textbook of Hartshorne \cite{03842033}. We include a short discussion for the convenience of the reader. Let $\mathcal{O}$ and $\mathcal{T}$ denote the structure sheaf and the tangent sheaf, respectively, of the smooth affine variety $X$ over the algebraically closed field $\mathbb{K}$. By $\mathfrak{i}(p_1, \dots, p_m) \subset \KK[X] = \mathcal{O}(X)$ we denote the ideal of the regular functions vanishing in the points $p_1, \dots, p_m \in X$. The sheaves $\mathcal{O}$ and $\mathcal{T}$ are coherent. By a classical result of Serre (see e.g.\ \cite{03842033}*{Theorem III.3.7}) for any exact sequence of sheaves of $\mathcal{O}$-modules 
$0 \to \mathcal{A} \to \mathcal{B} \to \mathcal{C} \to 0$ with $\mathcal{A}$ being coherent, the section functor is exact, i.e.\ $\mathcal{B}(X) \to \mathcal{C}(X)$ is surjective. In particular, we have the following exact sequences of sheaves of $\mathcal{O}$-modules:
\begin{align*}
    0 \to \mathfrak{i}^2(p_1, \dots, p_m) \mathcal{O} \to \mathcal{O} \to \mathcal{O}/\mathfrak{i}^2(p_1, \dots, p_m) \mathcal{O} \to 0 \\
    0 \to \mathfrak{i}^2(p_1, \dots, p_m) \mathcal{T} \to \mathcal{T} \to \mathcal{T}/\mathfrak{i}^2(p_1, \dots, p_m) \mathcal{T} \to 0 
\end{align*}
Note that $\mathcal{O}/\mathfrak{i}^2(p_1, \dots, p_m) \mathcal{O}$ is isomorphic to a finite direct sum of skyscraper sheaves that encode the value and the first derivative of a regular function on $X$ in the points $p_1, \dots, p_m$. 

Similarly, $\mathcal{T}/\mathfrak{i}^2(p_1, \dots, p_m) \mathcal{T}$ is isomorphic to a finite direct sum of skyscraper sheaves that encode the value and the first derivative of a vector field on $X$ in the points $p_1, \dots, p_m$. 

Since these two sheaves are also obviously coherent, so are the sheaves $\mathfrak{i}^2(p_1, \dots, p_m) \mathcal{O}$ and $\mathfrak{i}^2(p_1, \dots, p_m) \mathcal{T}$ as a consequence of the Three Lemma for sheaves. Now it follows that the section functor is exact for both sequences. Thus, we can always prescribe the value and derivative of a regular function or of a vector field in finitely many points on a smooth affine variety.

\begin{lemma}
    \label{lem-find-KK-pair}
    Let $p_1, \dots, p_m \in X$ be distinct points. Then there exists a pair of vector fields $(\delta, \eta)$ on $X$ and a function $h \in \KK[X]$ such that
    $(\delta(h) \eta + \eta(h) \delta)_{p_j} \neq 0$ for all $j = 1, \dots, m$.
\end{lemma}
\begin{proof}
    We choose $\delta = \eta$ such that $\delta_{p_j} \neq 0$ for all $j = 1, \dots, m$. Next, we choose $h \in \KK[X]$ such that $\delta(h)(p_j) = 1 \neq 0$ by prescribing the derivative of $h$ for all $j = 1, \dots, m$. 
\end{proof}

\begin{lemma}
    \label{lem-find-rotation}
     For every point $p \in X$, every vector $v \in T_p X$ and every vector field $\mu$ on $X$ with $\mu_p \neq 0$ there exists a vector field $\varepsilon$ on $X$ such that $[\varepsilon,\mu]_p = v$.
\end{lemma}
\begin{proof}
    First consider the germ in the point $p$. For this purpose we may assume without loss of generality that $X = \mathbb{A}^n$. In explicit coordinates we obtain
    \begin{equation}
        \label{eq-bracket}
        [\varepsilon,\mu] = \sum_{k,\ell=1}^n \left( a_\ell \frac{\partial b_k}{ \partial x_\ell} - b_\ell \frac{\partial a_k}{ \partial x_\ell}  \right) \frac{\partial}{\partial x_k}    
    \end{equation}    
    where $\varepsilon = \sum_{\ell=1}^n a_\ell \frac{\partial}{\partial x_\ell}$ and $\mu = \sum_{k=1}^n b_k \frac{\partial}{\partial x_k}$. 
    Thus, it is sufficient to prescribe the value and first derivative of $\varepsilon$ in the point $p$ in order to solve the linear system of equations in Equation \eqref{eq-bracket}, evaluated in the point $p$, for any $\mu$ with $\mu_p \neq 0$.    
\end{proof}

\begin{proof}[Proof of Theorem \ref{mainthm}]
We will proceed by induction from $j$ to $j-1$ where $j = 0, 1, \dots, n$. Set $A_n := X$.
Let $B_1, \dots, B_m$ be the irreducible components of $A_j$ and choose points $p_1 \in B_1, \dots, p_m \in B_m$. By Lemma \ref{lem-find-KK-pair} we find vector fields $\delta$ and $\eta$ and a regular function $h$ on $X$ such that $\theta := (\delta(h) \eta + \eta(h) \delta)$ does not vanish in the points $p_1, \dots, p_m$. By Equation \eqref{KK-formula} we can write every element in the module $\KK[X] \theta$ as a sum of two Lie brackets:
\[
    [f \delta, h \eta] - [h f \delta, \eta] = f \cdot \theta, \quad f \in \KK[X] \; .
\]
By Lemma \ref{lem-find-rotation} there exist vector fields $\varepsilon_1, \dots, \varepsilon_{n-1}$ on $X$ such that $\theta, [\varepsilon_1,\theta], \dots, [\varepsilon_{n-1},\theta]$ span the tangent spaces $T_{p_1} X, \dots, T_{p_m} X$. 

Noting that $[\varepsilon_k, f \theta] = f [\varepsilon_k, \theta] + \varepsilon_k(f) \cdot \theta$, we obtain every element in
\[
L_j := \KK[X] \theta + \KK[X] \cdot [\varepsilon_1, \theta] + \dots + \KK[X]  \cdot [\varepsilon_{n-1}, \theta]
\]
as a sum of $n+1$ brackets. 

Since $\theta, [\varepsilon_1,\theta], \dots, [\varepsilon_{n-1},\theta]$ span the tangent space in at least one point of each irreducible component of $A_j$, they in fact span the tangent spaces in a non-empty Zariski-open subset of each irreducible component of $A_j$.
Set
\[
A_{j-1} := \{ x \in A_j \;:\; \theta, [\varepsilon_1,\theta], \dots, [\varepsilon_{n-1},\theta] \text{ do not span } T_x X \}
\]
and proceed by induction. We have that $\dim A_{j-1} \leq \dim A_j - 1$ and thus this process terminates after at most $n+1$ steps. 

Consider the coherent subsheaves $\mathcal{L}_j \subset \mathcal{T}$ of the tangent sheaf that are generated by the $L_0,\dots,L_n$. The finite sum $\mathcal{H} = \mathcal{L}_{0} + \mathcal{L}_{1} + \dots + \mathcal{L}_{n}$ is then also coherent. Let $\mathfrak{m}_x$ be the maximal ideal of the point $x \in X$ in $\mathcal{O}(X)$. By construction, $\mathcal{H}/\mathfrak{m}_x \mathcal{H}$ coincides with $T_x X$ for every $x \in X$. Thus, $\mathcal{H} = \mathcal{T}$ by an application of the Nakayama lemma, see \cite{03842033}*{Exercise II.5.8} or, for an explicitly carried out computation, Leuenberger \cite{06596151}*{Lemma 2.1 and Lemma 2.2}. Thus, we obtain $\Ve(X) = {L}_{0} + {L}_{1} + \dots + {L}_{n}$ as a global section of $\mathcal{T}$.

For each of the at most $n+1$ summands we used at most $n+1$ brackets. 
\end{proof}

%\begin{remark}
%One can consider some variations of this proof, e.g.\ if $X$ is a smooth curve, one can choose the vector fields $\delta$ and $\eta$ such that together they span $T_x X$ in every point. In that case, $L_0$ and $L_1$ will be contained in a set generated by the difference of two brackets. However, $L_0 + L_1$ won't necessarily be contained in this set, thus not improving the bound for curves. 
%\end{remark}

\section{$\mathbb{G}_m$-varieties over characteristic $0$}

If an affine variety over a field $\KK$ admits a $\mathbb{G}_m$-action, i.e.\ an action of the group of units $\KK^\ast$, then we can usually give better bounds for the bracket width. 

\begin{lemma}
\label{lem-vecfieldcomb}
Let $X$ be an affine variety over an algebraically closed field $\KK$. Let $V_1, \dots, V_n$ be vector fields that span the tangent space $T_x X$ in every point $x \in X$. Then every polynomial vector field on $X$ can be written as a polynomial linear combination $f_1 V_1 + \dots + f_n V_n$ with $f_1, \dots, f_n \in \KK[X]$.
\end{lemma}
\begin{proof}
 This is an application of the Nakayama lemma, see \cite{03842033}*{Exercise II.5.8} or, for an explicitly carried out computation, Leuenberger \cite{06596151}*{Lemma 2.1 and Lemma 2.2}
\end{proof}

\begin{proposition}\label{boundforC*var}
Let $X$ be an affine variety over an algebraically closed field $\KK$ with $\Ch \KK = 0$. Let $x_1, \dots, x_N \in \KK[X] \setminus \KK$ such that $\KK[X] = \KK + (x_1, \dots, x_N) \KK[X]$. Let $V_1, \dots, V_n$ be vector fields that span the tangent space in every point of $X$. Let $H$ be a vector field that induces a $\mathbb{G}_m$-action and hence a $\ZZ$-grading on $\KK[X]$. Further assume that
\begin{enumerate}
\item 
\[
[H, V_k] = c_k V_k \quad \text{ with } c_k \in \ZZ \text{ for } k = 1, \dots, n \; .
\]
Let $M = \#\{ k \in \{1, \dots, n\} \,:\, c_k = 0 \}$ be the number of the vanishing constants $c_k$.
\item
\[
H(x_\ell) = d_\ell x_\ell \quad \text{ with } d_\ell \in \ZZ \setminus \{0\} \text{ for } \ell = 1, \dots, N \; .
\]
\item For each $k$ with $c_k = 0$ there exist vector fields $W_k, W'_k$ on $X$ such that $[W_k, W'_k] = V_k$.
\end{enumerate} 
If $V_k(x_\ell) \in (x_1, \dots, x_N) \KK[X]$ for all $k = 1, \dots, n$ and $\ell = 1, \dots, N$, then the bracket width of the Lie algebra of polynomial vector fields on $X$ is bounded by $1 + N + M$.

If $V_k(x_\ell) \notin (x_1, \dots, x_N) \KK[X]$ for some $k = 1, \dots, n$ and $\ell = 1, \dots, N$, then the bracket width of the Lie algebra of polynomial vector fields on $X$ is bounded by $1 + N + M + \Delta$ where $\Delta$ is the number of brackets needed to write $H \mod (x_1, \dots, x_N) \KK[X] \cdot H$.

\end{proposition}
\begin{proof}
Let $f_k, g, \widetilde{f}_{k \ell}, \widetilde{g}_\ell \in \KK[X]$ for $k = 1, \dots, n$ and $\ell = 1, \dots, N$.
Consider the sum of the following $1 + N$ brackets:
\begin{equation}
\label{eq-cstarterms}
\begin{split}
&[H, \sum_{k=1}^n f_k V_k + g H] + \sum_{\ell=1}^{N} [x_\ell H, \sum_{k=1}^n \widetilde{f}_{k\ell} V_k + \widetilde{g}_\ell H] \\
&= \left(H(f_k) + c_k f_k + \sum_{\ell=1}^{N} x_\ell (H(\widetilde{f}_{k\ell}) + c_k \widetilde{f}_{k\ell}) \right) V_k \\
&+ \left( H(g) + \sum_{\ell=1}^{N}\left(x_\ell H(\widetilde{g}_\ell) - x_\ell d_\ell \widetilde{g}_\ell - \sum_{k=1}^{n} \widetilde{f}_{k\ell} V_k(x_\ell) \right) \right) H.
\end{split}
\end{equation}
By Lemma \ref{lem-vecfieldcomb} we need to obtain every element of the form $\sum_{k=1}^n F_k V_k$ where $F_k \in \KK[X]$. Given any regular function $F_k$, we consider its decomposition as a sum of homogeneous functions with respect to the grading induced by $H$. Since the left-hand side is linear in $f_k, g, \widetilde{f}_{k \ell}, \widetilde{g}_l$, it is enough to treat the case when $F_k$ is homogeneous. Let $q \in \ZZ$ such that $H(F_k) = q F_k, q \in \ZZ$. We will achieve this in two steps: In the first step, we choose the functions $f_k$ and $\widetilde{f}_{k\ell}$ such that the coefficients in front of $V_k$ on right-hand side of Equation \eqref{eq-cstarterms} equal $F_k$ for each $k$. In the second step, we choose the functions $g$ and $\widetilde{g}_\ell$ such that the coefficients in front of $H$ vanish.
\begin{enumerate}
\item 
\begin{enumerate}
\item We first deal with the case $c_k \neq 0$.
\begin{enumerate}
    \item If $q \neq -c_k$, we set $f_k = F_k/(q + c_k)$ and $\widetilde{f}_{k \ell} = 0$. Since $c_k \neq 0$, this includes the constant functions $F_k$, which are of weight $0$. 
    \item If $q = -c_k$, then $F_k$ must be non-constant and we can factor out one of the $x_1, \dots, x_N$ and write $F_k = x_\ell \widetilde{F}_{k\ell}$. 
    Plugging this into $H(F_k) = q F_k$ yields $H(\widetilde{F}_{k \ell}) + c_{k} \widetilde{F}_{k \ell} = - d_\ell \widetilde{F}_{k \ell}$ with $d_\ell \neq 0$. Set $f_k = 0$ and $\widetilde{f}_{k \ell} = -\widetilde{F}_{k \ell}/d_\ell$.
\end{enumerate}
\item Next we consider the case $c_k = 0$.
\begin{enumerate}
    \item Again, we set $f_k = F_k/q$ and $\widetilde{f}_{k \ell} = 0$ if $q \neq 0$. However, in this case, the non-zero constant functions $F_k$ are not of this form. {\color{red} } 
    \item If $q = 0$, and if $F_k$ is not constant, then we can factor out one of the $x_1, \dots, x_N$ and write $F_k = x_\ell \widetilde{F}_{k\ell}$. We have $H(\widetilde{F}_{k \ell}) = - d_\ell \widetilde{F}_{k \ell}$ with $d_\ell \neq 0$. Set $f_k = 0$ and $\widetilde{f}_{jk} = - \widetilde{F}_{jk} / d_\ell$.
    \item In order to obtain the constant terms for $F_k$ when $q=0$, we need one additional bracket for each $k$ with $c_k = 0$: The missing constant term is provided by $[W_k, W'_k]$ which accounts for an additional $M$ brackets in total.
\end{enumerate}
\end{enumerate}
\item 
The coefficients in front of $H$ in Equation \eqref{eq-cstarterms} need to vanish, i.e.\ we need to choose $g$ and $\widetilde{g}_\ell$ such that
\begin{equation}\label{vanishcoeffinfrontofH}
H(g) + \sum_{\ell=1}^{N} (x_\ell H(\widetilde{g}_\ell) - x_\ell d_\ell \widetilde{g}_\ell ) = \sum_{\ell=1} ^{N} \sum_{k=1}^{n} \widetilde{f}_{k\ell} V_k(x_\ell).
\end{equation}
If $V_k(x_\ell) \in (x_1, \dots, x_N) \KK[X]$ for all $k = 1, \dots, n$ and $\ell = 1, \dots, N$, then the right-hand side contains no non-zero constant summands.
If $V_k(x_\ell) \notin (x_1, \dots, x_N) \KK[X]$ for some $k = 1, \dots, n$ and $\ell = 1, \dots, N$, then we may assume, after another additional number $\Delta$ of brackets, that the right-hand side contains no non-zero constant summands as well.

We carry out a similar computation as for the coefficients $F_k$ above. 
Assume $s \in \KK[X]$ is a function in the right-hand side of \eqref{vanishcoeffinfrontofH} which is homogeneous with respect to the grading induced by $H$ of non-zero degree. Then $s$ can be realized as $H(g)$ for some $g \in \KK[X]$ and hence a sum of such functions $s$ can also be realized as $H(g)$ for some $g \in \KK[X]$.
In particular, all linear terms $s$ in variables $x_1, \dots, x_N$ are obtained in this way.
To finish the proof, we need to show that each monomial $s \in \KK[X]$ of the right-hand side of \eqref{vanishcoeffinfrontofH}
with $H(s) = 0$ can be presented as $\sum_{\ell=1}^{N} (x_\ell H(\widetilde{g}_\ell) - x_\ell d_\ell \widetilde{g}_\ell )$ for some $\widetilde{g}_\ell \in \KK[X]$. We can assume that $s = x_\ell  h$ for some $\ell$ and some $h$. Since we have the equalities
\[
0 = H(s) = H(x_\ell  h) = H(x_\ell) h + x_\ell  H(h) = d_\ell x_\ell h +  x_\ell  H(h) 
\]
we conclude that $H(h) = -d_\ell h$. Note that by assumption $d_\ell\neq~0$. Now we set
  $\widetilde{g}_\ell = h/ (-2 d_\ell)$ and $\widetilde{g}_i = 0$, where $i \neq \ell$.  Then we get the equation
\begin{equation*}
s=x_\ell h = \sum_{i=1}^{N} (x_i H(\widetilde{g}_i) - x_i d_\ell \widetilde{g}_i ) \; .
\end{equation*}
 This finishes the proof. \qedhere
\end{enumerate}
\end{proof}

\begin{example}\label{korasrussel}
The Koras--Russell cubic threefold is given by the following equation in $\CC^4 \ni (x,y,z,w)$:
\[
x^2 y + x + z^2 +  w^3 = 0 \; .
\]
It is a well-known example of a smooth affine threefold that is diffeomorphic and even symplectomorphic to $\RR^6$ (see Casals and Murphy \cite{07036863}) but not algebraically isomorphic to $\CC^3$.
Let 
\[
H := 6x \frac{\partial}{\partial x} - 6y \frac{\partial}{\partial y} + 3z \frac{\partial}{\partial z} + 2w \frac{\partial}{\partial w}
\]
whose flow gives rise to a $\CC^\ast$-action on the Koras--Russell cubic threefold.

The following vector fields are tangential to the Koras--Russell cubic and span its tangent space in every point:
\begin{align*}
A &= 3w^2 \frac{\partial}{\partial z} - 2 z \frac{\partial}{\partial w}\\
B &= 3w^2 \frac{\partial}{\partial y} - x^2 \frac{\partial}{\partial w}\\
C &= 2z \frac{\partial}{\partial y} - x^2 \frac{\partial}{\partial z}\\
D &= -2z \frac{\partial}{\partial x} + (1 + 2xy) \frac{\partial}{\partial z}\\
E &= -x^2 \frac{\partial}{\partial x} + (1 + 2xy) \frac{\partial}{\partial y} \\
F &= -3w^2 \frac{\partial}{\partial x} + (1 + 2xy) \frac{\partial}{\partial w} \; .
\end{align*} 
With respect to the grading induced by $H$ they are of weight $10$, $9$, $-3$, $6$ and $-2$, respectively. The generators $x, y, z, w$ are of weight $6, -6, 3, 2$, respectively. We apply Proposition \ref{boundforC*var}. Due to the identity $x = -x^2y - z^2 - w^3$, we have $\CC[X] = \CC + (y,z,w)\CC[X]$. Hence $M=0$ and $N=3$. However, $D(z) = E(y) = F(w) = 1 + 2xy \notin (y,z,w)\CC[X]$. The other conditions are obviously satisfied. We compute the following sum of $3$ brackets
\begin{align*}
3[yC, D] -\frac{5}{2}[D, xD] + [A, wD] &= -6 y E + 5 z D + (-2 z D + 2 w F) \\
&= (1+2xy) H \\
&\equiv H \mod (y,z,w)\CC[X] \cdot H
\end{align*}
and thus obtain $\Delta = 3$. 
Therefore, the bracket width for the Koras--Russell cubic is a priori at most $7 = 1 + M + N + \Delta$.
\end{example}

\begin{question} \hfill
    \begin{enumerate}
        \item Is the Lie algebra $\Ve(X)$ for the Koras--Russell cubic threefold $X$ finitely generated?
        \item Is the bracket width of $\Ve(X)$ at least $2$? This would give another proof that the Koras--Russell cubic threefold is not algebraically isomorphic to $\CC^3$ since we know that the bracket width of $\Ve(\CC^3)$ is $1$.
    \end{enumerate}
\end{question}

\begin{example}\label{bracketwidthSLn}
The bracket width of $\Ve(\mathrm{SL}_n(\KK))$ over an algebraically closed field $\KK$ with $\Ch \KK = 0$ is at most $n^2 + n$:

Let $E_{ij} \in \mathrm{Mat}(n \times n; \KK)$ denote the elementary matrix with entry $1$ in the $i$-th row and $j$-th column and zeros elsewhere. We use the coordinates $(a_{ij})_{i,j=1}^{n}$ for $\mathrm{Mat}(n \times n; \KK) \supset \mathrm{SL}_n(\KK)$.

For $i \neq j$ we obtain the following well-known left-invariant vector fields that are LNDs:
\[
\left.\frac{d}{d t}\right|_{t = 0} \exp(t E_{ij}) =: V_{ij} = \sum_{m=1}^{n} a_{j m} \frac{\partial}{\partial a_{i m}} \; .
\]
For $k = 1, \dots, n-1$ we obtain the following well-known left-invariant vector fields whose flows are $\mathbb{G}_m$-actions:
\[
\left.\frac{d}{d t}\right|_{t = 0} \exp(t (-E_{kk}+E_{k+1,k+1})) =: H_{k} = \sum_{\ell=1}^{n} - a_{k \ell} \frac{\partial}{\partial a_{k \ell}} + a_{{k+1},\ell} \frac{\partial}{\partial a_{{k+1},\ell}} \; .
\]

A straightforward computation gives the following commutators:
 \begin{align*}
 [H_k, V_{k, k+1}] &=  2 V_{k, k+1} \\
 [H_k, V_{k+1,k}] &= -2 V_{k+1, k} \\
 [H_k, V_{k,   j}] &=    V_{k, j}, & j \notin \{k, k+1\} \\
 [H_k, V_{k+1, j}] &= -  V_{k+1, j}, & j \notin \{k, k+1\} \\
 [H_k, V_{i,   k}] &= -  V_{i, k}, & i \notin \{k, k+1\} \\
 [H_k, V_{i, k+1}] &=    V_{i, k+1}, & i \notin \{k, k+1\} \\
 [H_k, V_{i,   j}] &= 0, & i,j \notin \{k, k+1\} \\
 [H_1, H_k] &= 0,  & k = 1, \dots, n-1 \\
 [V_{ij},V_{pq}] &= \delta_{iq} V_{pj} - \delta_{pj} V_{iq} \\
 [V_{i,i+1},V_{i+1,i}] &= H_i \; .
 \end{align*}

We define the following vector field whose flow is also $\mathbb{G}_m$-action: 
\[
H := H_1 + 3 H_2 + 7 H_3 + 15 H_4 + \dots + (2^{n-1} - 1) H_{n-1}
\]
We now choose the $n^2-1$ left-invariant vector fields $V_{ij}$ with $1 \leq i \neq j \leq n$ and $H_k$ with $k=1, \dots, n-1$ that span the tangent space of $\mathrm{SL}_n(\KK)$ in every point. Observe that $[H, V_{ij}] = c_{ij} V_{ij}$ with $c_{ij} \neq 0$ and $[H,H_k] = 0$. Hence, we can apply Proposition \ref{boundforC*var} with $N = n^2$ and $M = n-1$.
\end{example}

\bigskip

For $\mathrm{SL}_2(\KK)$ we can further improve the bound on the bracket width.

\begin{proposition}\label{SL2generators}
The Lie algebra $\Ve(\mathrm{SL}_2(\KK))$ with $\Ch \KK = 0$ is generated by the following three (complete) vector fields: 
\[
V_{12}, a_{21} V_{21}, a_{22} V_{21} \; .
\]
\end{proposition}
\begin{proof}
By the result of the first author \cite{MR4588161}*{Theorem 14} we know that the four complete vector fields $V_{12}, V_{21}, (a_{12} + a_{21})V_{21}, a_{22}V_{21}$ generate the Lie algebra of polynomial vector fields on $\mathrm{SL}_2(\KK)$ (which was stated for $\KK = \CC$, but only $\Ch \KK = 0$ is used). To improve this result to three generators, we only need to observe that 
\[
[a_{22}V_{21}, a_{21}V_{21}] = (a_{22} V_{21}(a_{21}) - a_{21} V_{21}(a_{22})) V_{21} = V_{21}
\]
where we used the defining equation $a_{22} a_{11} - a_{21} a_{12} = 1$. Moreover, direct computations show that $[V_{21}, a_{22} V_{21}] = a_{12} V_{21}$. Hence,
\[
(a_{12} + a_{21})V_{21} = [V_{21}, a_{22} V_{21}] + a_{21})V_{21}.
\]
Therefore, we  obtain the four known generators as suitable Lie  combinations of $V_{12}, a_{21} V_{21}, a_{22} V_{21}$. 
\end{proof}

\begin{lemma}[Roman\cprime kov \cite{MR3601334}*{Lemma 1}]
\label{lem-generatorbound}
The number of generators of a Lie algebra is an upper bound for its bracket width.
\end{lemma}

\begin{corollary}\label{bracketboundSL2}
The bracket width of the Lie algebra of polynomial vector fields on $\mathrm{SL}_2(\KK)$ with $\Ch \KK = 0$ is at most $3$.
\end{corollary}

\begin{proof}
We use Lemma \ref{lem-generatorbound} to obtain the bound for the bracket width.
\end{proof}

\section{Funding}

The first author was supported by the European Union (ERC Advanced grant HPDR, 101053085 to Franc Forstneri\v{c}).

\end{document}